\documentclass[10pt,a4paper]{article}
\usepackage[latin1]{inputenc}
\usepackage{amsmath}
\usepackage{amsfonts}
\usepackage{amssymb}
\usepackage{euler}
\author{Olli Hella}
\title{On the characterization of $M$-metrizable spaces}
\date{27 June 2017}
\usepackage{amsthm}
\theoremstyle{plain}
\newtheorem{thm}{Theorem}[section] 
\theoremstyle{definition}
\newtheorem{defn}[thm]{Definition} 
\newtheorem{defns}[thm]{Definitions} 
\newtheorem{exmp}[thm]{Example}

\newtheorem{rmrk}[thm]{Remark}
\newtheorem{lmm}[thm]{Lemma}
\newtheorem{coroll}[thm]{Corollary}
\newtheorem{nttn}[thm]{Notation}
\newtheorem{empt}[thm]{}
\begin{document}
\maketitle
\noindent 
ABSTRACT.
We introduced the concept of a metric value set (MVS) in an earlier paper \cite{GM} and developed the idea further in \cite{AS}. In this paper we study locally $M$-metrizable spaces and the products of $M$-metrizable spaces. Finally we prove a characterization of  commutatively metrizable topological spaces by means of the bases of a quasiuniformity.

\section*{Introduction}
 
First we recall definitions from \cite{GM} and \cite{AS}. 
A \textit{metric value set} (MVS) is a set $M$ with at least two elements and a binary operation $+$ satisfying the following conditions:
\\
(M1) The operation $+$ is associative.
\\
(M2) The operation $+$ has a (then unique) neutral element $e$. Let $M^{*}=M\setminus\{e\}$.
\\
(M3) If $m_{1}+m_{2}=e$, then $m_{1}=m_{2}=e$.
\\
(M4) For every $m_{1},m_{2}\in M^{*}$ there are $m_{3}\in M^{*}$ and $m_{4},m_{5}\in M$ such that $m_{1}=m_{3}+m_{4}$ and $m_{2}=m_{3}+m_{5}$.

If the operation $+$ is also commutative, we say that $M$ is a \textit{commutative} MVS. 

The operation $+$ induces the relations $\unlhd$ and $\lhd$ in $M$ defined by setting $m_{1} \unlhd m_{2} $ if there is $m_{3}\in M$ such that $m_{1}+m_{3}=m_{2}$ and $m_{1}\lhd m_{2}$ if there is $m_{3}\in M^{*}$ such that $m_{1}+m_{3}=m_{2}$.

We say that $M$ is \textit{atom-free} if $M$ is commutative and for every $m\in M^{*}$ there exists $n\in M^{*}$ such that $n\lhd m$. We say that $M$ is \textit{strictly atom-free} if $M$ is commutative and for every $m\in M^{*}$ there exists $n\in M^{*}$ such that $n\lhd m \ntriangleleft n$.

 A \textit{quasimetric function} $f$ is a map $f\colon X\times X\rightarrow M$ from the square of a set $X$ to an MVS $M$ that satisfies the following conditions:
 \\
(f1) $f(x,z)\unlhd f(x,y)+f(y,z)$ for every $x,y,z\in X$ (triangle inequality).
\\
(f2) $f(x,x)=e$ for every $x\in X$.

A quasimetric function $f$ is a \textit{metric function} if it also satisfies the condition
\\
(f3) $f(x,y)=f(y,x)$ for every $x,y\in X$ (symmetry). 
 
 Let $f\colon X\times X\rightarrow M$ be a quasimetric function. Then the \textit{topology} $\mathcal{T}_{f}$ in $X$ \textit{induced by} $f$ is defined by the family
  \[
B_{f}(x,m)=\{y\in X: f(x,y)\lhd m\},\qquad m\in M^{*},
\]
of \textit{open balls} as the open neighbourhood base at $x$ for every $x\in X$. We call $(X,f)$ a \textit{quasimetric space} or an \textit{$M$-space} and denote $\textbf{B}_{f}\colon X\rightarrow \mathcal{P}(\mathcal{P}(X))$, $
\textbf{B}_{f}(x)=\{B_{f}(x,m): m\in M^{*}\}$, and call $\textbf{B}_{f}$ an open \textit{neighbourhood system}.

Let $(X,\mathcal{T})$ be a topological space. If there exists a quasimetric function $f\colon X\times X \rightarrow M$ 
such that the topology $\mathcal{T}_{f}$ induced by $f$ equals $\mathcal{T}$, then we say that $(X,\mathcal{T})$ or 
shortly $X$ is an $M$-\textit{metrizable} space. We say that $X$ is \textit{MVS-metrizable} if there is an MVS $M$ 
such that $X$ is $M$-metrizable. 

Let $(M_{1},+_{1})$ and $(M_{2},+_{2})$ be two metric value sets and $h\colon M_{1}\rightarrow M_{2}$ a map. We say that $h$ is a \textit{metric value set homomorphism} or shortly a \textit{homomorphism} if it satisfies the following two conditions:
\\
(H1) For all $m_{1}\in M$ it holds that $h(m_{1})=e_{M_{2}} \Longleftrightarrow m_{1}=e_{M_{1}}$.
\\
(H2) $h(m+_{1}n)=h(m)+_{2}h(n)$ for every $m,n\in M_{1}$.

 In this paper we define the concept of locally $M$-metrizable space in a natural way. Then we show that the topology of a metacompact locally $M$-metrizable space can be induced by a quasimetric function, when $M$ is atom-free. We also study shortly the topology of product spaces. Finally we show how to characterize the class of commutatively metrizable topological spaces by means of the bases of a quasiuniformity. 

\section{Metric functions and induced topology}
\begin{defn}
Let $X$ be a set, $\mathcal{T}$ a topology on $X$ and $\textbf{U}= \{U_{i}:i\in I\}$ a family of subsets of $X$. We say that $\textbf{U}$ is a \textit{generating set} for $\mathcal{T}$ or that it \textit{generates} $\mathcal{T}$ if $\mathcal{T}$ is the coarsest topology containing $\textbf{U}$. This is equivalent to saying that $\textbf{U}$ is a \textit{subbase} of $\mathcal{T}$. It is well known that the topology $\mathcal{T}$ generated by the subbase $\textbf{U}$ is the collection of unions of finite intersections of sets in $\textbf{U}$, i.e., the sets of the form  
\[
\bigcup_{j\in J}\bigcap_{k\in K_{j}} U_{jk} 
\]
where $J$ is an arbitrary index set, $K_{j}$ a finite index set for all $j\in J$ and $U_{jk}\in \textbf{U}$ for every $j\in J$ and $k\in K_{j}$.

Let $X$ be a set and let $(Y_{i},\mathcal{T}_{i})$ be a topological space and $f_{i}\colon X\rightarrow Y_{i}$ a map for every $i\in I$. The \emph{topology} of $X$ \emph{induced} by the family $(f_{i})_{i\in I}$ is generated by the subbase $\textbf{U}=\{f_{i}^{-1}U: i\in I, U\in \mathcal{T}_{i} \}$

\end{defn} 

\begin{lmm}\label{generating equivalence}
\textit{Let $X$ be a set and $\emph{\textbf{U}}$ a family of subsets of $X$. Let $\emph{\textbf{V}}\subset \emph{\textbf{U}}$ be such that every element of $\emph{\textbf{U}}$ is a union of some sets in $\emph{\textbf{V}}$. Then $\emph{\textbf{U}}$ and $\emph{\textbf{V}}$ generate the same topology.}$\qquad \square$
\end{lmm}
\noindent
Recall that a map $\textbf{B}\colon X\rightarrow \mathcal{P}(\mathcal{P}(X))$ with $\textbf{B}(x)\neq \emptyset$ for every $x\in X$ is a neighbourhood system for some topology on $X$ if it satisfies the following conditions:
\\
(B1) $x\in U$ for every $U\in \textbf{B}(x)$.
\\
(B2) If $U,V\in \textbf{B}(x)$, then there exists $W\in \textbf{B}(x)$ such that $W\subset U\cap V$.
\\
(B3) For every $U\in \textbf{B}(x)$ there exists $V\in \textbf{B}(x)$ such that for every $y\in V$ there exists $W\in \textbf{B}(y)$ such that $W\subset U$.

Let $\textbf{B}$ satisfy (B1), (B2) and the following condition:
\\
(B3') For every $U\in \textbf{B}(x)$ and $y\in U$ there exists $V\in \textbf{B}(y)$ such that $V\subset U$.
\\
Then we say that it is an open neighbourhood system. Since (B3) follows from (B3'), every open neighbourhood system is also a neighbourhood system.

\begin{defn}
We say that two neighbourhood systems $\textbf{B}$ and $\textbf{B}'$ on a set $X$ are \emph{equivalent} if they define the same topology. This is true when for every $x\in X$ and $B_{x}\in \textbf{B}(x)$ there exists a $B'_{x}\in \textbf{B}'(x)$ such that $B'_{x}\subset B_{x}$, and conversely.
\end{defn}

\begin{lmm}\label{generator set} 
\textit{Let $(Y_{i},\mathcal{T}_{i})$ be a topological space and $f_{i}\colon X\rightarrow Y_{i}$ a map for every $i\in I$. Let $\emph{\textbf{B}}_{i}$ be an open neighbourhood system for $Y_{i}$ for every $i\in I$. Then the topology induced by the family $(f_{i})_{i\in I}$ is generated by the subbase
\[
\emph{\textbf{V}}=\{f_{i}^{-1}V:i\in I, y_{i}\in Y_{i},V\in \emph{\textbf{B}}_{i}(y_{i})\}.
\]
}\begin{proof} Let $\mathcal{T}$ be the topology on $X$ induced by the maps $f_{i}$. The associated generating set of $\mathcal{T}$ is $\textbf{U}= \{f_{i}^{-1}U:i\in I, U\in \mathcal{T}_{i} \}$. Clearly $\textbf{V}\subset \textbf{U}$. Let $f_{i}^{-1}U\in \textbf{U}$. As an open set in $Y_{i}$, $U$ can be expressed as an union $\bigcup_{x\in U}B_{ix}$, $B_{ix}\in \textbf{B}_{i}(x)$. Thus $f_{i}^{-1}U=\bigcup_{x\in U}f_{i}^{-1}B_{ix}$. According to \ref{generating equivalence}, $\textbf{V}$ generates the same topology as $\textbf{U}$, which is the topology induced by the maps $f_{i}$.
\end{proof} 
\end{lmm}
\noindent
Let $X$ be a set, $Y_{i}$ be an $M_{i}$-metrizable space and $f_{i}\colon X\rightarrow Y_{i}$ a map for every $i\in I$. In some cases the topology induced by the family $(f_{i})_{i\in I}$ is induced by a quasimetric function. Next we study the case where $I$ is a singleton.

\begin{thm}
 Let $Y$ be an $M$-metrizable space and $f_{1}\colon Y\times Y\rightarrow M$ a quasimetric function which induces the topology of $Y$. Let $X$ be a set, $g\colon X\rightarrow Y$ a map from $X$ into $Y$ and $f_{2}\colon X\times X\rightarrow M$ the map defined by setting
\[
f_{2}(x_{1},x_{2})=f_{1}(g(x_{1}),g(x_{2})).
\]
Then $f_{2}$ is a quasimetric function and the topology $\mathcal{T}$ on $X$ induced by $g$ is the same as the topology $\mathcal{T}_{f_{2}}$ induced by $f_{2}$.

\begin{proof}

First we show that $f_{2}$ is a quasimetric function. 

Let $x_{1},x_{2},x_{3}\in X$. Now
\begin{align*}
f_{2}(x_{1},x_{3})=& f_{1}(g(x_{1}),g(x_{3}))
\\
\unlhd& f_{1}(g(x_{1}),g(x_{2})) +  f_{1}(g(x_{2}),g(x_{3}))
= f_{2}(x_{1},x_{2})+f_{2}(x_{2},x_{3}).
\end{align*} 
Thus (f1) holds. The condition (f2) holds, because $f_{2}(x,x)=f_{1}(g(x),g(x))=e$. Thus $f_{2}$ induces a topology $\mathcal{T}_{f_{2}}$ in $X$.

We show that $\mathcal{T}_{f_{2}}\subset\mathcal{T}$. For every $x\in X$ and $m\in M^{*}$ the following holds:
\begin{align*}
B_{f_{2}}(x,m)=&\{y\in X:f_{1}(g(x),g(y))\lhd m\}
\\
=&\{y\in X:g(y)\in B_{f_{1}}(g(x),m)\}
\\
=&g^{-1}B_{f_{1}}(g(x),m).
\end{align*}

To show that $\mathcal{T}\subset\mathcal{T}_{f_{2}}$ we need to prove that $\{g^{-1}B_{f_{1}}(y,m):y\in Y, m\in M^{*}\}$ which is a generating set of $\mathcal{T}$ by \ref{generator set}, is included in $\mathcal{T}_{f_{2}}$.

Let $V=B_{f_{1}}(y,m)$ and $x\in g^{-1}V$. Then $g(x)\in V$ and therefore $m_{1}=f_{1}(y,g(x))\lhd m$. Let $m_{2}\in M^{*}$ be such that $m_{1}+m_{2}=m$. Let $z\in B_{f_{2}}(x,m_{2})$. Then $f_{1}(g(x),g(z))=f_{2}(x,z)\lhd m_{2}$. Thus by triangle inequality $f_{1}(y,g(z))\lhd m_{1}+m_{2}=m$, and therefore $z\in g^{-1}V$. Hence $B_{f_{2}}(x,m_{2})\subset g^{-1}V$, which proves that $\mathcal{T}\subset\mathcal{T}_{f_{2}}$.
\end{proof}
\end{thm}

\begin{empt}\textbf{Relative topology.}\label{relative topology}
Let $X$ be a set with topology $\mathcal{T}$ induced by a quasimetric function $f_{1}\colon X\times X\rightarrow M$. Let $A\subset X$ and $i\colon A\hookrightarrow X$ be an inclusion. From the previous result we see that $f_{2}\colon A\times A\rightarrow M$, $f_{2}(a_{1},a_{2})=f_{1}(i(a_{1}),i(a_{2}))=f_{1}(a_{1},a_{2})$, is a quasimetric function, inducing on $A$ a topology $\mathcal{T}_{f_{2}}$, which is the same as the topology induced by $i$, i.e., the \textit{relative topology} on $A$. From the definition of $f_{2}$ we see that $f_{2}=f_{1}|_{A\times A}$. 

Furthermore we see that the property of being an $M$-metrizable space is a topological property. If $h\colon X\rightarrow Y$ is a homeomorphism and $Y$ is $M$-metrizable by a quasimetric function $f_{1}$, then we see that $X$ is $M$-metrizable by $f_{2}\colon X\times X\rightarrow M$ where $f_{2}(x_{1},x_{2})=f_{1}(h(x_{1}),h(x_{2}))$. 
\end{empt}

\begin{exmp}
Let $X$ be an uncountable set containing an element $x_{0}$ with the topology $\mathcal{T}$ a base of whose is formed by the one-point sets $\{x\}$, $x\neq x_{0}$, and those subsets of $X$ which contain $x_{0}$ and whose complement is finite, i.e., the example of a not MVS-metrizable space in [1, 1.12].

 Let $(M,+)=(\{0,1,2\},\max)$ and $\textbf{A}$ be the set whose elements are the sets $A\subset X$ such that $x_{0}\in A$ and $X\setminus A$ is finite. For every $A\in \textbf{A}$ we define a map $f_{A}\colon X\times X \rightarrow M$ by setting $f_{A}(x,y)=0$ if $x=y$ or $x,y\in A$ and $f_{A}(x,y)=2$ otherwise. 

The map $f=f_{A}$ is a quasimetric function. The condition (f2) holds by definition. Let $x,y,z\in X$. The triangle inequality $f(x,z)\unlhd f(x,y)+f(y,z)$ clearly holds if $f(x,z)=0$. Assume that $f(x,z)=2$. Then $x\notin A$ or $z\notin A$, and $x\neq y$ or $y\neq z$. Assume that $f(x,y)=0$. Then $x,y\in A$ or $x=y$. If $x,y\in A$, then $z\notin A$ and therefore $f(y,z)=2$. If $x=y$, then $z\neq y$ and $y=x\notin A$ or $z\notin A$, and thus $f(y,z)=2$. Hence $ f(x,y)+f(y,z)=\max\{f(x,y),f(y,z)\}=2$. Thus $f_{A}$ is a quasimetric function. It defines a topology $\mathcal{T}_{f_{A}}$ on $X$, whose base is formed by $A$ and the singletons $\{x\} $, $x\notin A$.

We can define a topology $\mathcal{T}'$ on $X$ induced by the maps $id_{A}\colon X\rightarrow (X,\mathcal{T}_{f_{A}})$, $x\mapsto x$, $A\in\textbf{A}$. The topology $\mathcal{T}'$ is generated by the sets $id_{A}^{-1}U=U$, $A\in \textbf{A}$, with $U$ belonging to the above base of $\mathcal{T}_{f_{A}}$. Thus a subbase of $\mathcal{T}'$ is $\textbf{A}\cup \bigcup_{x\neq x_{0}} \{\{x\}\}$. This subbase is actually the base of $\mathcal{T}$ above. Thus $\mathcal{T'}=\mathcal{T}$. Therefore the maps $id_{A}\colon X\rightarrow (X,\mathcal{T}_{f_{A}})$, $A\in \textbf{A}$, induce a topology on $X$ that is not MVS-metrizable, even though every $(X,\mathcal{T}_{f_{A}})$ is. Thus not every topological space induced by maps into MVS-metrizable spaces is itself MVS-metrizable. Whether there are counterexamples in the cases of finite or countably infinite number of inducing maps is an open question. However, if $M$ is an atom-free MVS, $X$ a topological space, $(Y_{i})_{i=1,...,n}$ a finite family of $M$-metrizable topological spaces, $g_{i}\colon X
\rightarrow Y_{i}$ a map for each $i=1,...,n$ and the topology of $X$ is induced by the family $(g_{i})_{i=1,..,n}$, then $X$ is $M$-metrizable. This follows from Theorems 
1.12 and 1.5: The topology of $X$ is induced by the map $g=(g_{1},...,g_{n})\colon X \rightarrow \prod_{i=1}^{n} Y_{i}$, since $g^{-1}\prod_{i=1}^{n} U_{i}=\bigcap_{i=1}^{n}
g_{i}^{-1}U_{i}$, when $U_{i}$ is open in $Y_{i}$ for each $i=1,...,n$.   
\end{exmp}

\begin{defn}
Let $X$ be a topological space and $M$ an MVS. If for every $x\in X$ there exists an open neighbourhood $U$ of $x$ whose relative topology is induced by a quasimetric function $f_{U}\colon U\times U\rightarrow M$, we say that $X$ is \textit{locally $M$-metrizable}.
 
Recall that a metacompact space is a topological space in which every open cover has an open refinement that is point-finite. 
\end{defn}

\begin{lmm}
\emph{Let $U=\{U_{i}:i\in I\}$ be an open cover of a topological space $(X,\mathcal{T})$. Let $\mathcal{T}_{U_{i}}$ be the relative topology of the set $U_{i}$ for every $i\in I$. Then for a set $V \subset X$ we have $V\in \mathcal{T}$ if and only if for every $x\in V$ and $i\in I$ with $U_{i}\ni x$ there exists $V_{i}\in \mathcal{T}_{U_{i}}$ such that $x\in V_{i}\subset V$.}
 
\begin{proof} 
If $V$ is open in $X$ and $x\in V$, then $V$ is a neighbourhood of $x$. If $x\in U_{i}$, then $U_{i}\cap V\subset V$ is a neighbourhood of $x$ which belongs to the relative topology $\mathcal{T}_{U_{i}}$.

If $V\subset X$ is not open, then there exists $x\in V$ such that every neighbourhood of $x$ intersects the complement of $V$. Let $i\in I$ be such that $x\in U_{i}$. Assume that there exists $V_{x}\in \mathcal{T}_{U_{i}}$ such that $x\in V_{x}\subset V$. Then $V_{x}$ is an open set in $X$ containing $x$. Hence $V_{x}$ is a neighbourhood of $x$ and thus contains at least one point that is not in $V$, which is a contradiction. Thus if $V\subset X$ is not open, then the condition does not hold.
\end{proof}
\end{lmm} 
\noindent
For the next theorem, recall $M_{\infty}$ from [1, 1.3.8].

\begin{thm} Let $M$ be an atom-free MVS and $(X,\mathcal{T})$ a metacompact locally $M$-metrizable topological space. Then $X$ is $M_{\infty}$-metrizable.
 \begin{proof} The $M$-metrizable open neighbourhoods $U_{x}$ of elements $x\in X$ form an open cover $\textbf{U}=\{U_{x}:x\in X\}$ of $X$. Let $f_{U_{x}}\colon U_{x}\times U_{x}\rightarrow M$ be a quasimetric function that defines the  relative topology on the neighbourhood $U_{x}$ for every $x\in X$. Now the topology $\mathcal{T}$ can be constructed from the relative topologies $\mathcal{T}_{f_{U_{x}}}$ as in the previous lemma.
 
Let $\textbf{V}=\{V_{j}:j\in J\} $ be a point-finite open refinement of $\textbf{U}$. The relative topology on $V_{j}$ can be induced by a quasimetric function $f_{V_{j}}\colon V_{j}\times V_{j}\rightarrow M$ for every $j\in J$. This is possible, because every $V_{j}$ is contained in some member $U_{x}$, $x\in X$, of $\textbf{U}$ and the inclusion $i\colon V_{j}\hookrightarrow U_{x}$ determines the relative topology on $V_{j}$ by the quasimetric function $f_{V_{j}}=f_{U_{x}}|_{V_{j}\times V_{j}}\colon V_{j}\times V_{j}\rightarrow M$. Furthermore  $\mathcal{T}_{f_{V_{j}}}$ is the relative topology also in terms of the whole space $X$.

For every $x\in X$, define the sets $J_{x}=\{j\in J:x\in V_{j}\}$ and
\[
W_{x}=\bigcap_{j\in J_{x}} V_{j}\ni x.
\]
The set $W_{x}$ is open, because it is a finite intersection of open sets. Furthermore, if $y\in W_{x}$, then $J_{x}\subset J_{y}$ and thus $W_{y}\subset W_{x}$. Now we can define a map $f\colon X\times X\rightarrow M_{\infty}$ in the following way:
\[
f(x,y)=\infty,\qquad \text{when } y\notin W_{x},
\]  
\[
f(x,y)=\sum_{j\in J_{x}}f_{V_{j}}(x,y),\qquad \text{when } y\in W_{x}.
\]
Notice that the sum is well-defined because $+$ is atom-free and hence commutative, and there is only a finite number of indices in $J_{x}$ for every $x\in X$. 
 
Clearly $f$ satisfies the condition (f2). Let $x,y,z\in X$. If $z\notin W_{x}$, then $y\notin W_{x}$ or $z\notin W_{y}$, because if neither holds, then $z\in W_{z}\subset W_{y}\subset W_{x}$, which is a contradiction. Thus we see that if $z\notin W_{x}$, then $f(x,y)=\infty$ or $f(y,z)=\infty$. Therefore the triangle inequality holds, when $z\notin W_{x}$.

Let $z\in W_{x}$. If $y\notin W_{x}$, then the triangle inequality holds. Let $y\in W_{x}$. If $z\notin W_{y}$, then the triangle inequality holds. Assume that $z\in W_{y}$. Then the triangle inequality holds in $W_{x}$ separately for every $f_{V_{j}}$; therefore from the commutativity of $+$ it follows that
\[
f(x,z)= \sum_{j\in J_{x}}f_{V_{j}}(x,z)\unlhd \sum_{j\in J_{x}}f_{V_{j}}(x,y)+ \sum_{j\in J_{x}}f_{V_{j}}(y,z).
\]   
Now $f(x,y)=\sum_{j\in J_{x}}f_{V_{j}}(x,y)$. Furthermore
\[
\sum_{j\in J_{x}}f_{V_{j}}(y,z)\unlhd \sum_{j\in J_{y}}f_{V_{j}}(y,z)=f(y,z).
\] 
These equations together imply that $f(x,z)\unlhd f(x,y)+f(y,z)$. Thus $f$ is a quasimetric function. Lastly we need to show that $\mathcal{T}_{f}=\mathcal{T}$.

Assume that $A\in \mathcal{T}$. Let $x\in A$. Now $x\in A\cap W_{x}\in \mathcal{T}$. Choose $j\in J_{x}$. Then $A\cap W_{x}\subset W_{x}\subset V_{j}$. Therefore $A\cap W_{x}\in \mathcal{T}_{f_{V_{j}}}$ and thus there exists $m_{j}\in M^{*}$ such that $B_{f_{V_{j}}}(x,m_{j})\subset A\cap W_{x}\subset A$. 

Let $y\in B_{f}(x,m_{j})$. Then $f(x,y)\lhd m_{j}$ and hence $f(x,y)\neq \infty$. Thus $y\in W_{x}\subset V_{j}$. Therefore $f_{V_{j}}(x,y)\unlhd f(x,y)\lhd m_{j}$. Thus $y\in B_{f_{V_{j}}}(x,m_{j})$ and therefore $y\in A$. Hence $A\in \mathcal{T}_{f}$. Thus $\mathcal{T}\subset \mathcal{T}_{f}$.

Let $x\in X$ and $m\in M_{\infty}^{*}$. Consider $B_{f}(x,m)\in \textbf{B}_{f}(x)$. If $m=\infty$, then $B_{f}(x,m)=X\in \mathcal{T}$. Assume that $m\neq \infty$. Let $n=\operatorname{card} J_{x}$. By atom-freeness of $M$ and [1, 2.10], choose $m_{n}\in M^{*}$ such that $n\cdot m_{n}\unlhd m$. Now the ball $B_{f_{V_{j}}}(x,m_{n})$ belongs to $ \mathcal{T}$ for every $j\in J_{x}$. Thus their intersection
\[
C=\bigcap\{B_{f_{V_{j}}}(x,m_{n}):j\in J_{x}\}
\]
also belongs to $\mathcal{T}$. Furthermore, if $y\in C$, then $f(x,y)\lhd n\cdot m_{n}\unlhd m$. Thus $C$ is a neighbourhood of $x$ (in $\mathcal{T})$ that is contained in the set $B_{f}(x,m)$. Therefore $\mathcal{T}_{f}\subset \mathcal{T}$ and thus $\mathcal{T}_{f}= \mathcal{T}$.  
\end{proof}
\end{thm}

\begin{rmrk} Even if the quasimetric functions $f_{U_{x}}$ are symmetrical, $f$ might not be, because it is possible that $y\in W_{x}$, but $x\notin W_{y}$. However, $f$ is strict (that is, $f(x,y)=e$ implies $x=y$), if every $f_{U_{x}}$ is strict.
\end{rmrk}

\begin{thm}[\textbf{Product of $M$-metrizable spaces}] Let $M$ be an atom-free $MVS$ and $\{X_{i}:i=1,2,...,n\}$ a finite collection of $M$-metrizable spaces. Then the product space $X=X_{1}\times ...\times X_{n}$ is $M$-metrizable.

\begin{proof}
For every $i\in \{1,2,...,n\}$ let $f_{i}\colon X_{i}\times X_{i}\rightarrow M$ be a quasimetric function that induces the topology of $X_{i}$. A neighbourhood base that defines the topology of $X$ is
\[
\textbf{B}(x)=\{B(x,(m_{1},...,m_{n})): m_{1},...,m_{n}\in M^{*} \},\qquad x\in X;
\]
\[
B(x,(m_{1},...,m_{n}))=\prod \{B_{f_{i}}(x_{i},m_{i}):i=1,...,n \},\qquad x=(x_{1},...,x_{n}).
\]
Define $f\colon X\times X\rightarrow M$, $f(x,y)=\sum_{i=1}^{n}f_{i}(x_{i},y_{i})$. Clearly $f$ satisfies the condition (f2). It also satisfies (f1), because $M$ is commutative and the quasimetric functions $f_{i}$ satisfy it. Thus $f$ is a quasimetric function. To prove the theorem, we show that $\textbf{B}$ and $\textbf{B}_{f}$ are equivalent.

Consider $B(x,(m_{1},...,m_{n}))\in\textbf{B}(x)$. Now there exists $m\in M^{*}$ such that $m\unlhd m_{i}$ for every $i=1,...,n$. Let $y\in B_{f}(x,m)$. Then
\[
f_{i}(x_{i},y_{i})\unlhd f(x,y)\lhd m\unlhd m_{i}
\]
for every $i=1,...,n$, and therefore $y\in B(x,(m_{1},...,m_{n}))$. Thus
\[
B_{f}(x,m)\subset B(x,(m_{1},...,m_{n})).
\]
Consider $B_{f}(x,m)\in \textbf{B}_{f}(x)$. By atom-freeness of $M$ we see that there exists $m'\in M^{*}$ such that $n\cdot m'\unlhd m$. Let $y\in B(x,(m',...,m'))$. Now
\[
f(x,y)=\sum_{i=1}^{n}f_{i}(x_{i},y_{i})\lhd n\cdot m'\unlhd m.
\]
Therefore $y\in B_{f}(x,m)$. Hence
\[
B(x,(m',...,m'))\subset B_{f}(x,m).
\]
The neighbourhood systems $\textbf{B}$ and $\textbf{B}_{f}$ are therefore equivalent, and thus $f$ induces the product topology on $X$. 
\end{proof}
\end{thm}

\begin{rmrk}
The set  $M=\{0,1,2\}$ with operation $+$ for which $x+0=x=0+x$ if $x\in M$ and  $x+y=2$ if
 $x\ne 0\ne y$ is a commutative MVS. The topology $\mathcal{T}_{f_M}$ on $M$  induced by the canonical metric function $f_M$ of $M$ (see [1, 1.11.5]) is discrete, and  clearly every $M$-metrizable Hausdorff space is discrete. Hence, the product of an infinite family
 of copies of the space $(M,\mathcal{T}_{f_M})$ is not $M$-metrizable. This example was found by Jouni Luukkainen.
\end{rmrk}
\noindent
The next theorem also found by Jouni Luukkainen states that every Alexandrov space, i.e., finitely generated space, is $M$-metrizable with strict quasimetric function, where $M$ is atom-free.    

\begin{thm} Let $M = \{0, 1, 2\}$ be the commutative, even atom-free, MVS whose operation is $(m, n) \mapsto \max\{m, n\}$. Let $(X, \mathcal{T})$ be a topological space such that every point $x$ of $X$ has a smallest neighbourhood $U(x)$, which is thus open in $X$. Such is the case if, for example, $\mathcal{T}$ or even $X$ is finite. Define a map $f \colon X \times X \rightarrow M$ by setting $f(x, x) = 0$ for each $x \in X$, $f(x, y) = 1$ if $y \in U(x)$ but $y \neq x$ and $f(x, y) = 2$ if $y \notin U(x)$. Then $f$ is a strict quasimetric function on $X$ inducing the topology of $X$, i.e., $\mathcal{T} = \mathcal{T}_{f}$.
\end{thm}
\begin{proof}
 First, $f$ is a quasimetric function, because if $x, y, z \in X$ with $x \neq y \neq z \neq x$ and $f(x, y) = f(y, z) = 1$,
then $y \in U(x)$ implying $U(y) \subset U(x)$ and $z \in U(y)$, whence $z \in U(x)$ and thus $f(x, z) = 1$. Consider $x \in X$. Then for each $y \in X$ with $y \neq x$ we have (note that $1 \lhd 1$)
\[
y \in B_{f} (x, 1) \Longleftrightarrow f(x, y) = 1 \Longleftrightarrow y \in U(x).
\]
Thus $U(x) = B_{f} (x, 1)$.

Finally, for a subset $U$ of $X$ we have
\[
U \in \mathcal{T} \Longleftrightarrow U =\bigcup_{x\in U} U(x) \Longleftrightarrow U = \bigcup_{x\in U} B_{f} (x, 1) \Longleftrightarrow U \in \mathcal{T}_{f} .
\]
Thus, $\mathcal{T} = \mathcal{T}_{f}$. 
\end{proof}

\section{Characterization results}
We want to characterize $MVS$-metrizable topological spaces. When an MVS $M$ is commutative, this can be done by using quasiuniform spaces.

\begin{nttn}
When $f\colon X\times X\rightarrow M$ is a metric function and $x\in X$, $m\in M^{*}$, we denote
\[
\bar{B}_{f}(x,m)=\{y\in X:f(x,y)\unlhd m\}
\]
and
\[
\bar{\textbf{B}}_{f}(x)=\{\bar{B}_{f}(x,m):m\in M^{*} \}.
\]
We denote the map that is defined by the collections $\bar{\textbf{B}}_{f}(x)$, $x\in X$, by $\bar{\textbf{B}}_{f}\colon X\rightarrow \mathcal{P}(\mathcal{P}(X)).$
\end{nttn}

\begin{thm} Let $M$ be an atom-free MVS and $f\colon X\times X\rightarrow M$ a quasimetric function. Then $\bar{\textbf{B}}_{f}$ is a neighbourhood system that is equivalent to the neighbourhood system $\textbf{B}_{f}$.
\end{thm}

 \begin{proof}
  Let $x\in X$. First, $B_{f}(x,m)\subset \bar{B}_{f}(x,m)$, if $m\in M^{*}$. 
  
  Let $m\in M^{*}$. Then by atom-freeness we can choose $m'\in M^{*}$ such that $m'\lhd m$. Thus $\bar{B}_{f}(x,m')\subset B_{f}(x,m)$.
  
  The map $\bar{\textbf{B}}_{f}$ satisfies the condition (B1), because $e\unlhd m$ for every $m\in M^{*}$. The condition (B2) follows from (M4).

  Let $U=\bar{B}_{f}(x,m)\in \bar{\textbf{B}}_{f}(x)$ and $m'\in M^{*}$ be such that $m'+m'\unlhd m$. Let $V=\bar{B}_{f}
  (x,m')\in \bar{\textbf{B}}_{f}(x)$. Let $y\in V$. Then $\bar{B}_{f}(y,m')$ is contained in $U$ due to the triangle 
  inequality. Therefore (B3) holds. Thus $\bar{\textbf{B}}_{f}$ is a neighbourhood system and it is equivalent to 
  $\textbf{B}.$ 
  \end{proof}
  
  \begin{rmrk}
  Notice that $\bar{\textbf{B}}_{f}$ is not necessarily an open neighbourhood system. For example, let $(X,d)$ be a metric space, $x\in X$ and $r\in ]0,\infty[$. Then the set $\bar{B}_{d}(x,r)$ is the closed ball $\{y\in X:d(x,y)\leq r\}$, which need not be an open set.
  \end{rmrk}
  
  \begin{empt}\textbf{Uniform spaces}.
  Recall the following definitions. Let $X$ be a set and $A,B\subset X\times X$. Define the set $A\circ B\subset X\times X$ by setting $(x,z)\in A\circ B$ if $x,z\in X$ and there exists $y\in X$ such that $(x,y)\in B$ and $(y,z)\in A$.
  
  Let $X$ be a set and $\mathcal{U}$ a non-empty family of subsets of $X\times X$. The pair $(X,\mathcal{U})$ is a \textit{uniform space} and $\mathcal{U}$ a \emph{uniformity} if $\mathcal{U}$ satisfies the following conditions:
  \\
  (U1) If $U\in \mathcal{U}$, then $U$ contains the \textit{diagonal} $\Delta=\{(x,x):x\in X\}$ of $X$.
  \\
  (U2) If $U\in \mathcal{U}$ and $U\subset V\subset X\times X$, then $V\in \mathcal{U}$.
   \\
  (U3) If $U,V\in \mathcal{U}$, then $U\cap V\in \mathcal{U}$.
   \\
  (U4) If $U\in \mathcal{U}$, then there exists $V\in \mathcal{U}$ such that $V\circ V\subset U$.
   \\
  (U5) If $U\in \mathcal{U}$, then $U^{-1}=\{(y,x):(x,y)\in U\}\in \mathcal{U}$.
  \\
  If the last condition is omitted, then $(X,\mathcal{U})$ is a \textit{quasiuniform space} and $\mathcal{U}$  a \emph{quasiuniformity.}
  
  Let $(X,\mathcal{U})$ be a quasiuniform space. A collection $\mathcal{U}_{\textbf{B}}\subset \mathcal{P}(X\times X)$ is a \emph{base of the quasiuniform space} $X$, if $\mathcal{U}_{\textbf{B}}\subset \mathcal{U}$ and for every $U\in \mathcal{U}$ there exists $U_{B}\in \mathcal{U}_{\textbf{B}}$ such that $U_{B}\subset U$. 
  
The base defines the quasiuniform space uniquely: a set $U\subset X\times X$ belongs to $\mathcal{U}$ if and only if there exists $U_{B}\in \mathcal{U_{B}}$ that is contained in $U$.
\end{empt}

\begin{thm}
Let $X$ be a set and $\mathcal{U}_{\textbf{B}}$ a non-empty collection of subsets of $X\times X$. Then $\mathcal{U}_{\textbf{B}}$ is a base of some quasiuniform space if and only if it satisfies the following conditions:
\\
\\
\emph{(UB1)} $\Delta \subset U$ for every $U\in \mathcal{U}_{\textbf{B}}$. 
\\
\emph{(UB2)} If $U,V\in \mathcal{U}_{\textbf{B}}$, then there exists $W\in \mathcal{U}_{\textbf{B}}$ such that $W\subset U\cap V$.
\\
\emph{(UB3)} If $U\in \mathcal{U}_{\textbf{B}}$, then there exists $V\in\mathcal{U}_{\textbf{B}}$ such that $V\circ V\subset U$.
\end{thm}

\begin{proof}
First assume that $\mathcal{U}_{\textbf{B}}$ is a base of a quasiuniformity $\mathcal{U}$. The condition (UB1) clearly holds. Let $U,V\in \mathcal{U}_{\textbf{B}}$. Then $U,V\in \mathcal{U}$, and therefore $U\cap V \in \mathcal{U}$. Thus there exists $W\in \mathcal{U}_{\textbf{B}}$ that is contained in $U\cap V$. Therefore (UB2) holds. The condition (UB3) follows easily from (U4).

Let $\mathcal{U}_{\textbf{B}}$ be a collection satisfying (UB1)--(UB3). Let 
\[
\mathcal{U}=\{U:U\subset X\times X, U_{B}\subset U \text{ for some } U_{B}\in \mathcal{U}_{\textbf{B}}\}.
\]
We will show that $\mathcal{U}$ is a quasiuniformity. The condition (U1) follows from (UB1) and (U3) from (UB2). The condition (U2) holds, because every $U\in \mathcal{U}$ contains some $W\in \mathcal{U}_{\textbf{B}}$ and if $V\supset U$, then $V$ contains $W$ and is therefore a member of $\mathcal{U}$. The condition (U4) is satisfied by (UB3). Now $\mathcal{U}_{\textbf{B}}$ is a base of $(X,\mathcal{U})$ by the definition of $\mathcal{U}$.
\end{proof}

\begin{empt} \textbf{The topology induced by a quasiuniform space.}
Let $(X,\mathcal{U})$ be a quasiuniform space and $\mathcal{U}_{\textbf{B}}$ a base of $\mathcal{U}$. For $x\in X$ and $U\in \mathcal{U}_{\textbf{B}}$ let
\[
U[x]=\{y\in X:(x,y)\in U\}.
\]
For $x\in X$ define a non-empty family
\[
\textbf{B}_{\mathcal{U}_{\textbf{B}}}(x)=\{U[x]:U\in \mathcal{U}_{\textbf{B}}\}
\]
of subsets of $X$.

Now we show that $\textbf{B}_{\mathcal{U}_{\textbf{B}}}$ is a neighbourhood system on $X$. The condition (B1) holds, because $(x,x)\in U$ and thus $x\in U[x]$ for every $U\in \mathcal{U}$ due to (UB1). Let $U[x],V[x]\in \textbf{B}_{\mathcal{U}_{\textbf{B}}}(x)$. Due to (UB2), there exists $W\in \mathcal{U}_{\textbf{B}}$ such that $W\subset U\cap V$. Therefore $W[x]\subset U[x]\cap V[x]$. Thus (B2) holds.

Let $U[x]\in \textbf{B}_{\mathcal{U}_{\textbf{B}}}(x)$. Due to (UB3) there exists $V\in \mathcal{U}_{\textbf{B}}$ such that $V\circ V\subset U$. Let $y\in V[x]$ and $z\in V[y]$. Then $(x,y)\in V$ and $(y,z)\in V$. Hence $(x,z)\in V\circ V\subset U$ and thus $z\in U[x]$. Therefore (B3) holds.

We denote the topology defined by the neighbourhood system $\textbf{B}_{\mathcal{U}_{\textbf{B}}}$ by $\mathcal{T}
_{\mathcal{U}_{\textbf{B}}}$ and call it the \emph{topology induced by the base $\mathcal{U}_{\textbf{B}}$}
of the quasiuniform space $(X,\mathcal{U})$. 
If $(X,\mathcal{T})$ is a topological space with $\mathcal{T}=\mathcal{T}_{\mathcal{U}_{\textbf{B}}}$, we say that the quasiuniformity that $\mathcal{U}_{\textbf{B}}$ induces is compatible with $\mathcal{T}$. 
\end{empt}

\begin{lmm}
\textit{If $\mathcal{U}_{\textbf{B}_{1}}$ and $\mathcal{U}_{\textbf{B}_{2}}$ are bases of a quasiuniform space $(X,
\mathcal{U})$, then $\mathcal{T}_{\mathcal{U}_{\textbf{B}_{1}}}=\mathcal{T}_{\mathcal{U}_{\textbf{B}_{2}}}$.} $\square$
\end{lmm}

\begin{empt}\label{M as Q}\textbf{$M$-metrizable spaces as quasiuniform spaces.}
Let $M$ be an atom-free MVS and $f\colon X\times X\rightarrow M$ a quasimetric function. Define the set
\[
U_{m}=\{(x,y)\in X\times X:f(x,y)\unlhd m\} 
\]
for every  $m\in M^{*}$. Now the collection $\mathcal{U}_{M^{*}}=\{U_{m}:m\in M^{*}\}$ is a base of a quasiuniform space. The condition (UB1) holds, since $f(x,x)=e\unlhd m$ for every $x\in X$ and $m\in M^{*}$. Let $m_{1},m_{2}\in M^{*}$. Choose $m_{3}\in M^{*}$ such that $m_{3}\unlhd m_{1},m_{2}$. Now $U_{m_{3}}\subset U_{m_{1}}\cap U_{m_{2}}$. Therefore (UB2) holds.

Let $m_{1}\in M^{*}$. Due to atom-freeness there exists $m_{2}\in M^{*}$ such that $m_{2}+m_{2}\unlhd m_{1}$. Let $(x,z)\in U_{m_{2}}\circ U_{m_{2}}$. Then there exists $y\in X$ such that $(x,y)\in U_{m_{2}}$ and $(y,z)\in U_{m_{2}}$. Therefore $f(x,y)\unlhd m_{2}$ and $f(y,z)\unlhd m_{2}$, and from the triangle inequality it follows that $f(x,z)\unlhd m_{2}+m_{2}\unlhd m_{1}$. Hence $U_{m_{2}}\circ U_{m_{2}}\subset U_{m_{1}}$. Thus (UB3) holds and $\mathcal{U}_{M^{*}}$ is a base of some quasiuniform space $(X,\mathcal{U})$. 

We now show that the topology $\mathcal{T}_{\mathcal{U}_{M^{*}}}$ induced by $\mathcal{U}_{M^{*}}$ is the same as the topology $\mathcal{T}_{f}$ induced by $f$. This follows from the fact that the neighbourhood systems $\bar{\textbf{B}}_{f}$ and $\textbf{B}_{\mathcal{U}_{M^{*}}}$ are the same, because
\[
U_{m}[x]=\{y\in X:f(x,y)\unlhd m\} = \bar{B}_{f}(x,m). \qquad\square
\]

\noindent
We also notice that if $f$ is symmetrical, then $U_{m}=U_{m}^{-1}$ for every $m\in M^{*}$ implying that $\mathcal{U}_{M^{*}}$ is a base of a uniformity.
\end{empt}
\begin{defns}
Let $M$ be a commutative MVS. We need two new concepts to characterize all $M$-metrizable spaces in terms of bases of quasiuniform spaces.

We say that an $M$-metrizable space $(X,f)$ is a \emph{full $M$-space} and that $f$ is \emph{$M$-full}, if $f$ is a surjection. If $(X,\mathcal{T})$ is a topological space whose topology is given by some surjective quasimetric function $f\colon X\times X\rightarrow M$, then we say that $(X,\mathcal{T})$ or shortly $X$ is an \emph{$M$-full space}. Notice that there might be two different quasimetric functions that induce the same topology such that only one of them is $M$-full.

Let $f\colon X\times X\rightarrow M$ be a quasimetric function which satisfies the following condition:
\\
(C) If $m_{1},m_{2},m_{3}\in M$ are such that $m_{2}+m_{3}=m_{1}$ and $x,y\in X$ such that $f(x,y)=m_{1}$, then there exists $z\in X$ such that $f(x,z)=m_{2}$ and $f(z,y)=m_{3}$.

Then we say that $(X,f)$ is a \emph{convex $M$-space} and that $f$ is \emph{$M$-convex}. If $(X,\mathcal{T})$ is a topological space whose topology is given by some $M$-convex $f$, then we say that $(X,\mathcal{T})$ or shortly $X$ is an \emph{$M$-convex space}. 
\end{defns}
 
 \begin{lmm} \textit{Let $M$ be a commutative MVS. Then every $M$-space can be embedded into an $M$-full space.}
 \begin{proof}
 An empty $M$-space can be embedded into the example space $(M,f_{M})$ in [1, 1.11.5]. Let $(X,f)$ then be a non-empty $M$-space. We form a product set $X\times M$ and a map
 \[
 f_{*}\colon (X\times M)\times(X\times M)\rightarrow M
 \]
 such that $f_{*}((x,m),(x,m))=e_{M}$ for every $x\in X$, $m\in M$, and
 \[
 f_{*}((x_{1},m_{1}),(x_{2},m_{2}))=m_{1}+f(x_{1},x_{2})+m_{2},
 \]
 when $x_{1}\neq x_{2}$ or $m_{1}\neq m_{2}$.
 
 Now $f_{*}$ clearly satisfies the condition (f2). Let $(x_{1},m_{1}),(x_{2},m_{2}),(x_{3},m_{3})\in X\times M$. Let $(x_{1},m_{1})\neq (x_{3},m_{3})$. Then
 \begin{align*}
  f_{*}((x_{1},m_{1}),(x_{3},m_{3}))&= m_{1}+f(x_{1},x_{3})+m_{3}
\\
&\unlhd m_{1}+(f(x_{1},x_{2})+f(x_{2},x_{3}))+m_{3}
\\
&\unlhd f_{*}((x_{1},m_{1}),(x_{2},m_{2}))+f_{*}((x_{2},m_{2}),(x_{3},m_{3})),
 \end{align*}
 which can be checked straightforwardly. Therefore $f_{*}$ satisfies the triangle inequality and is thus a quasimetric function. Furthermore if $m\in M$ and we choose $x\in X$, then $f_{*}((x,m),(x,e_{M}))=m$. Therefore $f_{*}$ is $M$-full.
When $X$ and $X\times \{e_{M}\}$ are identified, then the restriction of $f_{*}$ to $X$ is $f$. Therefore the inclusion $X\hookrightarrow X\times M$ is an embedding of $X$ into an $M$-full space $X\times M$.  
\end{proof}   
\end{lmm}

\begin{lmm} \textit{Let $M$ be a commutative MVS. Then every $M$-space $X$ can be embedded into an $M$-convex space.}
\begin{proof}
Denote $0=e_{M}$. We introduce some notation. If $(X,f)$ is an $M$-space, then we denote by $X_{f}$ the subset of $X\times M\times M\times X$ consisting of the elements $(x_{1},m_{1},m_{2},x_{2})$ such that $m_{1}+m_{2}=f(x_{1},x_{2})$. Furthermore we identify $x$ and $(x,0,0,x)$. Intuitively $X_{f}$ is constructed from the points in $X$ and from the new points lying between two points in $X$. 

Let $(X,f_{1})$ be an $M$-space. Define a map $f_{2}\colon X_{f_{1}}\times X_{f_{1}}\rightarrow M$ by setting
\[
f_{2}((x_{1},m_{1},m_{2},x_{2}),(x_{3},m_{3},m_{4},x_{4}))=0,
\] 
when $(x_{1},m_{1},m_{2},x_{2})=(x_{3},m_{3},m_{4},x_{4})$, and otherwise
\[
f_{2}((x_{1},m_{1},m_{2},x_{2}),(x_{3},m_{3},m_{4},x_{4}))=m_{2}+f(x_{2},x_{3})+m_{3}.
\]
It can be straightforwardly checked that $(X_{f_{1}},f_{2})$ is an $M$-space and furthermore $f_{2}((x,0,0,x),(y,0,0,y))=f_{1}(x,y)$. In the same way we can define an $M$-space $((X_{f_{1}})_{f_{2}},f_{3})$ and so on. Denote $X_{1}=X$, $X_{2}=X_{f_{1}}$, $X_{3}=(X_{f_{1}})_{f_{2}}$ and so on. Considering the identifications, we see that $X_{1}\subset X_{2}\subset X_{3}\subset ...$ . Let $X_{\omega}=\bigcup_{n\in \mathbb{N}} X_{n}$. Define $f_{\omega}\colon X_{\omega}\times X_{\omega}\rightarrow M$ by setting $f_{\omega}(x,y)=f_{\min_{x,y}}(x,y)$, when $x,y\in X_{\omega}$ and $\min_{x,y}\in\mathbb{N}$ is the smallest index $k$ such that both $x$ and $y$ belong to the set $X_{k}$.

Let $x,y,z\in X_{\omega}$. Now
\begin{align*}
f_{\omega}(x,y)+f_{\omega}(y,z)&=f_{\min_{x,y}}(x,y)+f_{\min_{y,z}}(y,z)
\\
&=f_{\max\{\min_{x,y},\min_{y,z}\}}(x,y)+ f_{\max\{\min_{x,y},\min_{y,z}\}}(y,z)
\\
&\unrhd f_{\max\{\min_{x,y},\min_{y,z}\}}(x,z)=f_{\omega}(x,z),
\end{align*}
because every $f_{k}$ is a quasimetric function and $f_{k}(x,y)=f_{l}(x,y)$ always, when $x,y\in X_{k}\cap X_{l}$. 
Therefore $f_{\omega}$ satisfies the condition (f1). The condition (f2) follows directly from the definition of $f_{\omega}$. Therefore $f_{\omega}$ is a quasimetric function. Furthermore $X=X_{1}\subset X_{\omega}$ and $f_{1}$ is the restriction of $f_{\omega}$ to the set $X_{1}$. Thus $X$ can be embedded into the $M$-space $X_{\omega}$.

The space $X_{\omega}$ is $M$-convex: Let $x,y\in X_{\omega}$, $f_{\omega}(x,y)=m+n$ and $k,l\in \mathbb{N}$ indexes such that $x\in X_{k}$ and $y\in X_{l}$. We may assume that $m\neq 0\neq n$. Now $(x,m,n,y)\in X_{\max\{k,l\}+1}$ and 
\[
f_{\omega}((x,0,0,x),(x,m,n,y))=0+f_{\omega}(x,x)+m=m
\]
and
\[
f_{\omega}((x,m,n,y),(y,0,0,y))=n+f_{\omega}(y,y)+0=n.
\]
Therefore $(X_{\omega},f_{\omega})$ is an $M$-convex space such that $X$ can be embedded into it.
\end{proof}
\end{lmm}
 \begin{coroll}\label{full and convex}
 \textit{Let $M$ be a commutative MVS. Then every $M$-space $X$ can be embedded into an $M$-full and $M$-convex space $\tilde{X}$.}
 \begin{proof}
 Let $M$ be commutative and $X$ an $M$-space. Then $X$ can be embedded into an $M$-full space $Y$ and $Y$ into an $M$-convex space $\tilde{X}=Y_{\omega}$. Clearly $Y_{\omega}\supset Y$ is also $M$-full, and thus $X$ can be embedded into the $M$-full and $M$-convex space $\tilde{X}$.
 \end{proof}
 \end{coroll}
 
 \begin{thm}\label{Q theorem}
 Let $M$ be an atom-free MVS and $(X,f)$ an $M$-full and $M$-convex space. Then there exists a set $U_{0}\subset X\times X$ that contains the diagonal $\Delta$ and a compatible quasiuniformity $\mathcal{U}$ on $X$ that has a base $\mathcal{V}^{*}$ such that $U_{0}\subsetneq U$ for all $U\in \mathcal{V}^{*}$, $\bigcup \mathcal{V}^{*}=X\times X$ and that satisfies the conditions below, when we define for all $x,y\in X$ a subset $U_{x,y}\ni (x,y)$ of $X\times X$ by setting
 \[
 U_{x,y}=\bigcap\{U\in\mathcal{V}^{*}:(x,y)\in U\}, \qquad \text{if } (x,y)\notin U_{0};
 \]
 \[
 U_{x,y}=U_{0}, \qquad \text{if } (x,y)\in U_{0}.
 \]
 \emph{(Q1)} The base $\mathcal{V}^{*}$ has a representation as a family $\{U_{x,y}:(x,y)\in (X\times X)\setminus U_{0}\}.$
 \\
 \emph{(Q2)} The following operation $+$ can be defined in the set $\mathcal{V}=\mathcal{V}^{*}\cup \{ U_{0} \}$:
 
 The sum $U+V$ is the intersection of the sets $W\in \mathcal{V}$ such that $V\circ U\subset W$. The set $\mathcal{V}$ is closed under the operation $+$ and furthermore the pair $(\mathcal{V},+)$ is an MVS with neutral element $U_{0}$.
 \\
 \emph{(Q3)} When a relation $\unlhd$ is defined by $+$ as usually, then $U\unlhd V$ if and only if $U\subset V$.
  \\
 \emph{(Q4)} The MVS $(\mathcal{V},+)$ is atom-free.  
 \begin{proof}
 Let $(X,f)$ be an $M$-full and $M$-convex space. Then by \ref{M as Q} the topology of $X$ is defined by the quasiuniformity $\mathcal{U}$ that has a base $\mathcal{V}^{*}$ whose elements are the sets
 \[
 U_{m}=\{ (x,y):f(x,y)\unlhd m\}, \qquad m\in M^{*}.
 \]
 
 Let $0=e_{M}$. Define $U_{0}=\{ (x,y):f(x,y)=0\}$, from which it follows that $\{(x,x):x\in X\}\subset U_{0}$ due to (f2). Let $m\in M^{*}$. Since $X$ is $M$-full, it holds that $m=f(x,y)$ for some $x,y\in X$. Now $f(x,y)\neq 0$. Thus $(x,y)\in U_{m}\setminus U_{0}$ and hence $U_{0}\subsetneq \mathcal{U}_{m}$. Clearly $\bigcup \mathcal{V}^{*}= X \times X$. Next we will show that $\mathcal{V}$ satisfies the conditions (Q1)--(Q4).

 Let $x,y\in X$ and $f(x,y)\neq 0$. Now $(x,y)\in U_{f(x,y)}\in \mathcal{V}^{*}$ and therefore $U_{x,y}\subset U_{f(x,y)}$.
 
 Let $(a,b)\in U_{f(x,y)}$ and $f(x,y)\neq 0$. Then $f(a,b)\unlhd f(x,y)$. Thus if $(x,y)\in U_{m}$, i.e., $f(x,y)\unlhd m$, then $(a,b)\in U_{m}$. Therefore $(a,b)\in U_{x,y}$. Hence $U_{f(x,y)}\subset U_{x,y}$ and so $U_{f(x,y)}=U_{x,y}$. Then especially (Q1) holds, because $(X,f)$ is $M$-full.

 Define an operation $+$ as in the condition (Q2). Let $U_{m},U_{m'}\in \mathcal{V}$. We will show that $U_{m}+U_{m'}=U_{m+m'}$.
 
 From the triangle inequality it follows that $U_{m'}\circ U_{m}\subset U_{m+m'}$. Therefore $U_{m}+U_{m'}\subset U_{m+m'}$.
 
 Let $n\in M$ be such that $U_{m'}\circ U_{m}\subset U_{n}$. Since $(X,f)$ is $M$-full, there exist $x,y\in X$ such that  $f(x,y)=m+m'$. Since $(X,f)$ is $M$-convex, there exists $z\in X$ for which $f(x,z)=m$ and $f(z,y)=m'$. Now $(x,z)\in U_{m}$ and $(z,y)\in U_{m'}$ and thus $(x,y)\in U_{m'}\circ U_{m}\subset U_{n}$. Thus $m+m'=f(x,y)\unlhd n$, and so $U_{m+m'}\subset U_{n}$. Therefore $U_{m+m'}\subset U_{m}+U_{m'}$. Thus $U_{m}+U_{m'}=U_{m+m'}\in \mathcal{V}$. The set $\mathcal{V}$ is therefore closed under the operation $+$.

 The pair $(\mathcal{V},+)$ is an MVS. The associativity, that $U_{0}$ is a neutral element, and the conditions (M3) and (M4) follow from the facts that $U_{m}+U_{m'}=U_{m+m'}$, that $U_{0}\notin \mathcal{V}^{*}$, and that $M$ is an MVS. Therefore (Q2) holds. 

Next we show that $U_{m}\subset U_{m'}$ if and only if $m\unlhd m'$. Assume that $U_{m}\subset U_{m'}$. Now $m=f(x,y)$ for some $x,y\in X$, because $X$ is $M$-full. Thus $(x,y)\in U_{m}\subset U_{m'}$, and so $m=f(x,y)\unlhd m'$. The other direction is obvious.

Let $U_{m}\unlhd U_{m'}$. Choose an $m''\in M$ such that $U_{m}+U_{m''}=U_{m'}$. Then $U_{m}\subset U_{m+m''}=U_{m}+U_{m''}=U_{m'}$, since $m\unlhd m+m''$.

Let $U_{m}\subset U_{m'}$. Then $m\unlhd m'$. Therefore there exists $m''\in M$ such that $m+m''=m'$. Thus $U_{m}+U_{m''}=U_{m+m''}=U_{m'}$. Therefore $U_{m}\unlhd U_{m'}$. Thus (Q3) holds.

The map $h\colon M\rightarrow \mathcal{V}$, $m\mapsto U_{m}$, is a surjective homomorphism. Now $\mathcal{V}$ is commutative, since $M$ is commutative. Let $U\in \mathcal{V}^{*}$. Then $U=U_{m}$ for some $m\in M^{*}$. Since $M$ 
is atom-free there exists $n\in M^{*}$ such that $n+n\unlhd m$. Let $V=U_{n}$. Then $V\in \mathcal{V}^{*}$ and $V+V=U_{n+n}\unlhd U_{m}=U$. Hence $\mathcal{V}$ is atom-free, i.e., (Q4) holds.
\end{proof}
 \end{thm}

\begin{thm}\label{QT theorem 2}
Let $(X,\mathcal{U})$ be a quasiuniform space for which there exist a set $U_{0}$ and a base $\mathcal{V^{*}}$  satisfying the conditions \emph{(Q1)--(Q3)}. Then $X$ is a $\mathcal{V}$-metrizable topological space.
\begin{proof}
Define a map $f\colon X\times X\rightarrow \mathcal{V}$ by setting $f(x,y)=U_{x,y}$. Due to the condition (Q2), the pair $(\mathcal{V},+)$ is an MVS and furthermore $f(x,x)=U_{x,x}$ is its neutral element.

Let $x,y,z\in X$. Since $(x,z)\in U_{y,z}\circ U_{x,y}$ it follows that $(x,z)\in U_{x,y}+U_{y,z}$. Thus $U_{x,z}\subset U_{x,y}+U_{y,z}$. Therefore $U_{x,z}\unlhd U_{x,y}+U_{y,z}$ by (Q3), and thus $f$ is a quasimetric function. Denote the topology induced by $f$ by $\mathcal{T}_{f}$  and the topology given by the quasiuniformity $\mathcal{U}$ by $\mathcal{T}_{\mathcal{U}}$.

The neighbourhood bases defining $\mathcal{T}_{f}$ and $\mathcal{T}_{\mathcal{U}}$ are the same, because for every $x\in X$ and $U\in \mathcal{V}^{*}$ we have
\begin{align*}
U[x]&=\{y\in X:(x,y)\in U\}=\{ y\in X:U_{x,y}\subset U\}
\\
&=\{y\in X:f(x,y)\subset U\}=\{ y\in X:f(x,y)\unlhd U\}
\\
&=\bar{B}_{f}(x,U).
\end{align*}
Thus the topologies are the same.
\end{proof}
\end{thm}

\begin{empt}\textbf{Summary of the results.}
By \ref{full and convex} and \ref{Q theorem} every topological space $X$ that is $M$-metrizable with an atom-free $M$ can be embedded into a quasiuniform space $(X,\mathcal{U})$ that has a base that satisfies the conditions (Q1)--(Q4).

By \ref{QT theorem 2} and \ref{relative topology}, every topological space that can be embedded into a quasiuniform space that has a base satisfying the conditions (Q1)--(Q4) is $M$-metrizable with some atom-free $M$. We say that a topological space $X$ is \emph{atom-freely metrizable} if it is $M$-metrizable with some atom-free $M$. Thus we have characterized the class of atom-freely metrizable spaces by means of quasiuniform spaces.

Correspondingly we say that a topological space $X$ is \emph{commutatively metrizable} (\emph{strictly atom-freely metrizable}) if it is $M$-metrizable with a commutative (strictly atom-free) MVS $M$. From [1, 2.15] it follows that the class of commutatively metrizable spaces is the same as the class of strictly atom-freely metrizable spaces and thus the same as the class of atom-freely metrizable spaces. If $M$ in \ref{Q theorem} is strictly atom-free, then also $(\mathcal{V},+)$ is strictly atom-free; for otherwise by [1, 2.11] there would be $m_{0}\in M^{*}$ such that $U_{m_{0}}\unlhd U_{m}$ for every $m\in M^{*}$, from which it would follow that $m_{0}\unlhd m$ for every $m\in M^{*}$, which is a contradiction.  
\end{empt}
\section*{Acknowledgement}
The author wants to thank Jouni Luukkainen for thorough revision of this paper, correcting grammar mistakes, making helpful suggestions and clarifying proofs to make this paper more readable.

\end{document}